\documentclass[10pt]{amsart}
\usepackage{amsmath}
\usepackage{enumerate}
\usepackage[colorlinks=true]{hyperref}
\usepackage{cleveref}

\usepackage{esint}
\newtheorem{theorem}{Theorem}[section]
\newtheorem{corollary}[theorem]{Corollary}
\newtheorem{lemma}[theorem]{Lemma}
\newtheorem{proposition}[theorem]{Proposition}

\theoremstyle{definition}
\newtheorem{definition}[theorem]{Definition}
\theoremstyle{remark}
\newtheorem{remark}{Remark}
\numberwithin{equation}{section}

\def\D{\nabla}
\newcommand{\abs}[1]{\left\vert#1\right\vert}
\newcommand{\RR}{\mathbb{R}}

\newcommand{\NN}{\mathbb{N}}

\newcommand{\XX}{W^{1,k+1}(B_1,\abs{x}^{1-k})}
\newcommand{\norm}[1]{\left\Vert#1\right\Vert}
\newcommand{\dive}{{\rm div}}
\newcommand{\Q}{\mathcal{Q}}

\usepackage[notref,notcite]{showkeys}

\begin{document}
\title[A characterization of semistable radial solutions of $k$-{H}essian]{A characterization of semistable radial solutions of $k$-{H}essian equations}
\author{Miguel Angel Navarro and Justino S\'anchez}
\address{Departamento de Estat\'{i}stica, An\'alise Matem\'atica e Optimizaci\'on\\Universidade de Santiago de Compostela\\Santiago de Compostela 15782, Spain.}
\email{miguelangel.burgos@usc.es }
\address{Departamento de Matem\'{a}ticas\\
	Universidad de La Serena\\
	Avenida Cisternas 1200, La Serena, Chile.}	
\email{jsanchez@userena.cl}	
\maketitle	
\begin{abstract}
We characterize semistable radial solutions of the equation
\begin{equation*}
S_k\left(D^2u\right)=g(u)\;\; \mbox{in }\;B_1,
\end{equation*}
where $B_1$ is the unit ball of $\RR^n$, $D^2u$ is the Hessian matrix of $u,\,g$ is a positive $C^1$ nonlinearity and $S_k\left(D^2u\right)$ denotes the $k$-Hessian operator of $u$. This class of radial solutions has been recently introduced by the authors in \cite{NaSa19}. The proofs are new relative to those given in \cite{NaSa19} and focus on the structure of the equation directly, thereby improving some previous results.
\end{abstract}

2010 Mathematics Subject Classification 35J60, 35J25 (primary), 35B35, 35B07 (secondary).

\section{Introduction and statement of results}
In this work we are concerned with the following nonlinear equation
\begin{equation}\label{Eq:f}
S_k\left(D^2u\right)=g(u)\;\; \mbox{in }\;B_1.
\end{equation}

Here $B_1$ is the unit ball of $\RR^n$, $D^2u$ is the Hessian matrix of $u,\,g$ is a positive $C^1$ nonlinearity and $S_k\left(D^2u\right)$ denotes the $k$-Hessian operator of $u$.

In our previous paper \cite{NaSa19} we introduced the notion of semistability for solutions of equation \eqref{Eq:f} under a homogeneous Dirichlet boundary condition. In the radial case, sharp pointwise estimates on semistable solutions were obtained, extending some results from the semilinear case, i.e., when $k=1$. However, to work with the new notion of stability for solutions of \eqref{Eq:f} even under rotational symmetry conditions it was necessary to introduce the auxiliary equation
\begin{equation}
\dive\left(\abs{x}^{1-k}\abs{\D u}^{k-1}\D u\right)=c_{n,k}^{-1}\,g(u)\mbox{ in }B_1\setminus\{0\},
\label{Eq:P}
\refstepcounter{equation}
\tag{$P$}
\end{equation}
($c_{n,k}=\binom{n}{k}/n$) and exploit the fact that, for radial solutions, both equations coincide. The main purpose of this paper is to characterize the class of radially symmetric solutions which are semistable, in a suitably-defined sense, for equation \eqref{Eq:f}. For this, we use new arguments based on the radial structure of \eqref{Eq:f} in order to remove equation \eqref{Eq:P} in \cite{NaSa19}. 

For $k\in\{1,...,n\}$, let $\sigma_k:\RR^n\rightarrow \RR$ denote the $k$-th elementary symmetric function
\[
\sigma_k\left(\lambda\right)=\sum_{1\leq i_1<...<i_k\leq n}{\lambda_{i_1}\cdots\lambda_{i_k}},
\]
and let $\Gamma_k$ denote the cone $\Gamma_k=\{\lambda=(\lambda_1,...,\lambda_n): \sigma_1(\lambda)>0,...,\sigma_k(\lambda)>0\}$. For a twice differentiable function $u$ defined on a smooth domain $\Omega\subset\RR^n$, the {\it $k$-Hessian operator} is defined by  
$
S_k\left(D^2u\right)=\sigma_k\left(\lambda\left(D^2u\right)\right),
$
where $\lambda\left(D^2u\right)$ are the eigenvalues of $D^2u$. Equivalently, $S_k\left(D^2u\right)$ is the sum of the $k$-th principal minors of the Hessian matrix. See e.g. \cite{Wang09, Wang94}. Two relevant examples in this family of operators are the Laplace operator $S_1\left(D^2u\right)=\Delta u$ and the Monge-Amp\`{e}re operator $S_n\left(D^2u\right)=\mbox{det}\left(D^2u\right)$. They are fully nonlinear when $k\geq 2$. In particular,  $S_2\left(D^2u\right)=\frac{1}{2}\left((\Delta u)^2-\abs{D^2 u}^2\right)$. The study of $k$-Hessian equations have many applications in geometry, optimization theory and other related fields. See \cite{Wang09}. These operators have been studied extensively, starting with the seminal work of Caffarelli, Nirenberg and Spruck \cite{CaNS85}. See, e.g., \cite{ChWa01, Jacobsen99, Jacobsen04, JaSc02, Tso89, Tso90}. 

Now consider the class of functions 
\[
\Phi^k\left(\Omega\right)=\left\lbrace u\in C^2\left(\Omega\right)\cap C\left(\overline{\Omega}\right): \lambda\left(D^2u\right)\in\Gamma_i,\, i=1,\ldots,k\right\rbrace .
\]

The functions in $\Phi^k\left(\Omega\right)$ are called {\it admissible} or $k$-{\it convex functions}. Further, $S_k\left(D^2u\right)$ turns to be elliptic in the class of $k$-convex functions. Denote by $\Phi_0^k\left(\Omega\right)$ the set of functions in $\Phi^k\left(\Omega\right)$ that vanish on the boundary $\partial\Omega$. Observe that the functions in $\Phi_0^k\left(\Omega\right)$ are negative in $\Omega$. For more details we refer the reader to \cite{Wang09}.
 
The following two notions of solutions for problem \eqref{Eq:f} were introduced recently in \cite{NaSa19}
\begin{definition}\label{def:solutiongen}
	We say that:
	\begin{enumerate}[$i)$]
		\item $u$ is a {\it classical solution} of \eqref{Eq:f} if $u\in \Phi_0^k(B_1)$ and equation \eqref{Eq:f} holds;
		\item $u$ is a {\it weak solution} of \eqref{Eq:f} if $u\in L^{k+1}(B_1)$, $\int_{B_1}{\left\lbrace \sum{u_iu_jS_k^{ij}(D^2u)}\right\rbrace}<\infty$, $g(u)\in L^1(B_1)$ and
		
		\begin{equation*}
		\int_{B_1}{\left\lbrace\frac{1}{k}\sum{u_i\eta_jS_k^{ij}(D^2u)}+g(u)\eta\right\rbrace }=0,\,\forall \eta\in C_c^{1}(B_1),
		\end{equation*}
		where $u_i=u_{x_i}$ indicates the partial derivative of $u$ with respect to the variable $x_i$.
	\end{enumerate}
	
 We recall that $S_k^{ij}(D^2u)=\frac{\partial}{\partial u_{ij}}S_k(D^2 u)$, where $u_{ij}=u_{x_ix_j}$. This expression is related to the divegence structure of the $k$-Hessians, $S_k(D^2 u)=\frac{1}{k}\sum (u_j S_k^{ij}(D^2u))_i$. For instance, when $k=1$, we have $S_1^{ij}(D^2u)=\delta_{ij}$ and $S_1(D^2 u)=\sum\delta_{ij}u_{ij}$, where $\delta_{ij}$ is the Kronecker delta symbol.
\end{definition}

\begin{definition}\label{def:semistablegen} 
	Let $u$ be a solution of \eqref{Eq:f}. We say that $u$ is {\it semistable} if
	\begin{equation}\label{eq:matchalQgen}
	\mathcal{Q}_u(\varphi)=\int_{B_1}{\left\lbrace\sum{\varphi_i\varphi_jS_k^{ij}(D^2u)}+g'(u)\varphi^2\right\rbrace }\geq 0,\,\forall \varphi\in C_c^1\left(B_1\setminus\{0\}\right).
	\end{equation}
\end{definition}

From a variational point of view, semistable solutions of equation \eqref{Eq:f} in $\Phi_0^k\left(B_1\right)$ correspond to critical points of an energy functional with nonnegative second variation \eqref{eq:matchalQgen} (in particular, local minimizers of the energy are semistable solutions). See \cite{Tso90, Wang94}.

Recently, in \cite{WaLe19}, the authors gave a definition of (classical) stable radial solutions of the $k$-Hessian equation 
$F_{k}(D^2V)=(-V)^p$ in $\RR^n$. They stablished connections between certain critical exponents of Joseph-Lundgren type and stability. We point out that our definition of semistability was motivated by the variational structure of equation in \eqref{Eq:f} and the fact that the $k$-Hessians can be written in divergence form. Note that our semistability condition \eqref{eq:matchalQgen} agrees (with the obvious changes) with the one given in \cite{WaLe19} if $u$ is radial. See \eqref{ineq:propstable} below. Furthermore, we obtain explicitly $S_k^{ij}(D^2u)$ in terms of the eigenvalues of the Hessian matrix of $u$, which is key for characterizing the semistable solutions (classical or weak solutions).

We recall that, for a radially symmetric $C^2$ function $u$, the $k$-Hessian operator is given by
\begin{equation}\label{Radial:Hess}
S_k(D^2 u)=c_{n,k}r^{1-n}\left (r^{n-k}(u')^k\right )'=c_{n,k}\left(\frac{u'}{r}\right)^{k-1}\left(n\left(\frac{u'}{r}\right)+k\left(u''-\frac{u'}{r}\right)\right),
\end{equation}
where $u(x)=u(r),\, r=\abs{x}, x\in\RR^n\setminus\{0\}$ and $c_{n,k}=\binom{n}{k}/n$. Here $u'$ denote the radial derivative of the radial function $u$. This formula is well known. For self-containment, we include a proof of \eqref{Radial:Hess} in Section 5 (the reader who is familiar with \eqref{Radial:Hess} may certainly skip this proof). 

\begin{remark}\label{rema:equiv}
Using the sign condition on $g$, it is easy to see that the following statements are equivalent: $(a)$ $u$ is a classical radial solution of \eqref{Eq:f}; $(b)$ $u$ is a $C^2$ solution of $c_{n,k}r^{1-n}(r^{n-k}(u')^k)'=g(u),\, r\in (0,1)$ satisfying $u'(0)=u(1)=0$. In particular, for a classical radial solution $u$ of \eqref{Eq:f}, the above equivalence implies that $u'(r)>0$ for all $r\in (0,1)$. 
\end{remark}

Our main results are the following two theorems which characterize weak and semistable radial solutions of \eqref{Eq:f}

\begin{theorem}\label{theo:solweak}
	Let $g(u)\in L^1(B_1)$. A function $u\in\XX$ is a weak radial solution of \eqref{Eq:f} if and only if
	
	\begin{equation*}
	\int_{B_1}{\left\lbrace c_{n,k}\abs{x}^{-k}(u')^{k}\left(x,\D \xi\right)+g(u)\xi\right\rbrace }=0,
	\end{equation*}
	for every radially symmetric function $\xi\in C_c^{1}(B_1)$.	
\end{theorem}

\begin{theorem}\label{theo:esential}
	A function $u$ is a semistable radial solution of \eqref{Eq:f} if and only if
	
	\begin{equation}\label{ineq:propstable}
	\Q_u(\xi):=\int_{B_1}{kc_{n,k}\abs{x}^{1-k}(u')^{k-1}\abs{\nabla\xi}^2+g'(u)\xi^2}\geq 0,
	\end{equation}
	for every radially symmetric function $\xi\in C_c^{1}(B_1)$.
\end{theorem}

We also have the following notion of solution for equation \eqref{Eq:f}

\begin{definition}\label{def:integralsolution} 
	An absolutely continuous function $u$ defined on $(0,1]$ is an {\it integral radial solution} of \eqref{Eq:f} if\, $r^{n-1}g(u)\in L^{1}(0,1),\, \int_0^1r^{n-1}\abs{u}^{k+1}<\infty,\,\int_{0}^{1}r^{n-k}(u')^{k+1}dr<\infty$ and
\begin{equation}\label{integraldef}
r^{n-k}(u')^k=c_{n,k}^{-1}\int_{0}^{r}s^{n-1}g(u)\,ds\mbox{ a.e. in }(0,1).
\end{equation}
\end{definition}

Thus, in the context of radial solutions, we can use Theorem \ref{theo:solweak} to show that the definitions of weak solution and integral solution are equivalent. 
\begin{lemma}\label{equiv}
Let $u$ be a weak radial solution of \eqref{Eq:f}. Then $u$ is an integral radial solution, and conversely.
\end{lemma}

A consequence of the previous lemma is the following statement concerning regularity of the solutions.  
\begin{corollary}\label{coro:weakC2}
Let $u$ be an integral radial solution of \eqref{Eq:f}. Then $u\in C^2(0,1]$.
\end{corollary}
In the following statement, the semistability inequality \eqref{eq:matchalQgen} is rewriting in a form that makes it independent of $g$. 
\begin{corollary}\label{coro:esential}
A function $u$ is a semistable radial solution of \eqref{Eq:f} if and only if
\begin{equation}\label{goeto0}
\begin{split}
\int_{B_1}{\left(\frac{u'}{\abs{x}}\right)^{k+1}\left\lbrace\abs{\abs{x}\D\eta}^2+\left(\frac{k-1}{k+1}\right)\left(x,\D\eta^2\right)-\left(\frac{2n-k-1}{k+1}\right)\eta^2\right\rbrace}\geq 0,
\end{split}
\end{equation}
for every radially symmetric function $\eta\in H_c^1\left(B_1\right)\cap L_{loc}^{\infty}\left(B_1\right)$.
\end{corollary}
The preceding result is necessary for characterizing the semistable solutions.
\medskip

The key point in the proof of Theorems \ref{theo:solweak} and \ref{theo:esential} is to obtain explicitly $S_{k}^{ij}(D^2 u)$ in terms of the eigenvalues of $D^2 u$. To this end, we take advantage of the radial structure of equation \eqref{Eq:f} and thus avoid to use of equation \eqref{Eq:P}, which was a technical requirement in \cite{NaSa19}. We also employ some ideas from \cite{MR2476421} and \cite{FaNa20}. It is worth noting that the main arguments in our proofs appear to be new.

As a consequence of our results, we obtain the following pointwise estimates for solutions of equation \eqref{Eq:f}, which are analogues of the corresponding results in \cite[Section 2]{NaSa19}. 

\begin{theorem}\label{th_1}
	Let $n\geq 2$, $g:\RR\to\RR$ be a nonnegative locally Lipschitz function and $u\in\XX$ be a semistable radial solution of equation \eqref{Eq:f}.
	Then the following holds:
	\begin{enumerate}[$i)$]
		
		\item If $n<2k+8$, then
		\[\abs{u(r)} \leq C,\;\forall r\in (0,1].
		\]		
		\item If $n=2k+8$, then
		\[\abs{u(r)} \leq C\left (\abs{\log r} +1\right ),\;\forall r\in (0,1].
		\]
		\item If $n>2k+8$, then 
		\[
		\abs{u(r)} \leq Cr^{\frac{-(k+1)n+2\sqrt{2(k+1)n-4k}+2k^2+6k}{(k+1)^2}},\;\forall r\in (0,1].
		\]			
	\end{enumerate}
	
	Here $C=D_{n,k}\norm{u}_{W^{1,k+1}(B_1\setminus\overline{B_{1/2}},\abs{x}^{1-k})}$ where $D_{n,k}$ is a constant depending only on $n$ and $k$.
\end{theorem}

\begin{theorem}\label{thm:est_g}
	Let $n\geq 2k+8$, $g:\RR\to\RR$ be a locally Lipschitz function and $u\in\XX$ be a semistable radial solution of equation \eqref{Eq:f} in $B_1$.
	Then the following holds:
	\begin{enumerate}[$i)$]
		\item If $g$ is a nonnegative function, then
		\[
		u'(r) \leq D_{n,k}\norm{\D u}_{L^{k+1}(B_1\setminus\overline{B_{1/2}},\abs{x}^{1-k})}r^{\frac{-(k+1)n+2\sqrt{2(k+1)n-4k}+k^2+4k-1}{(k+1)^2}},\;\forall r\in (0,1/2].
		\]			
		\item If $g$ is a nonnegative and nonincreasing function, then
		\[
		\abs{u^{(i)}(r)} \leq D'_{n,k}\left(\min_{t\in[1/2,1]} u'(t)\right)r^{\frac{-(k+1)n+2\sqrt{2(k+1)n-4k}+(2-i)k^2+2(3-i)k-i}{(k+1)^2}},\;\forall r\in (0,1],\,i\in\{1,2\}.
		\]			
		\item If $g$ is a nonnegative, nonincreasing and convex function, then 
		\[
		\abs{u^{(3)}(r)} \leq D'_{n,k}\left(\min_{t\in[1/2,1]} u'(t)\right)r^{\frac{-(k+1)n+2\sqrt{2(k+1)n-4k}-k^2-3}{(k+1)^2}},\;\forall r\in (0,1].
		\]	
	\end{enumerate}
	
	Here $D_{n,k}$ and $D'_{n,k}$ are constants depending only on $n$ and $k$.
\end{theorem}

\begin{remark}
	Similar to \cite{Villegas12}, in the last section we show that, without making assumptions on the sign of $g'$ or $g''$, it is impossible to obtain any pointwise estimates for $\abs{u''}$ and $\abs{u'''}$ (see Corollaries \ref{cor:inequrr} and \ref{cor:inequrrr}). 
\end{remark}

\begin{theorem}\label{th_2}
	Let $n\geq 2$, $g:\RR\to\RR$ be a nonnegative and nonincreasing locally Lipschitz function and $u\in\XX$ be a semistable radial solution of equation \eqref{Eq:f}.
	Then the following holds:
	\medskip
	
	\begin{enumerate}[$i)$]
		\item If $n<2k+8$, then 
		\[
		\abs{u(r)-u(1)}\leq C(1-r),\,\forall r\in(0,1].
		\]
		\item If $n=2k+8$, then 
		\[
		\abs{u(r)-u(1)}\leq C\abs{\log r},\,\forall r\in(0,1].
		\]
		\item If $n>2k+8$, then 
		\[
		\abs{u(r)-u(1)}\leq C\left (r^{\frac{-(k+1)n+2\sqrt{2(k+1)n-4k}+2k^2+6k}{(k+1)^2}}-1\right ),\,\forall r\in(0,1].
		\]
	\end{enumerate}
	
	Here $C=D_{n,k}\min_{t\in[1/2,1]}{\{u'(t)\}}$ where $D_{n,k}$ is a constant depending only on $n$ and $k$.	
\end{theorem}

\begin{remark}
Note that in the above statements no mention is made of equation \eqref{Eq:P}. Compare with Theorems 2.5, 2.6 and 2.8 in \cite{NaSa19}.
\end{remark}

The rest of this paper is organized as follows: in Section 2 we will state preliminaries which include new insights into the radial structure of equation \eqref{Eq:f}, which are key to proving our main results. Section 3 will then be devoted to the proof of Theorems \ref{theo:solweak} and \ref{theo:esential}, Lemma \ref{equiv} and Corollaries \ref{coro:weakC2} and \ref{coro:esential}. In Section 4 we provide a large family of semistable radially increasing unbounded solutions of problem \eqref{Eq:f} in the punctured unit ball. Section 5, which concludes the paper, contain a derivation of the radial form of the $k$-Hessians operators.

\section{Preliminaires}
In this section we compute $S_{k}^{ij}(D^2 u)$ for a radially symmetric $C^2$ function $u$. As observed above, we exploit the radial structure of equation \eqref{Eq:f} to obtain explicitly $S_{k}^{ij}(D^2 u)$ in terms of the eigenvalues of $D^2 u$, which is fundamental for characterizing the solutions. This is one of the novelties of this paper. To this end, set $x=\left(x_1,\ldots,x_n\right)\in\RR^n\setminus\{0\}$. Then, for a radial function $u=u(r),\, r=\abs{x}$, we have
\begin{equation}\label{eq:uij}
u_i=x_i\lambda_2\mbox{ and }u_{ij}=\delta_{ij}\lambda_2+\frac{x_{i}x_{j}}{\abs{x}^2}\left(\lambda_1-\lambda_2\right),
\end{equation}	
where $\lambda_1=u''$ and $\lambda_2=u'/\abs{x}$ are the eigenvalues of $D^2 u$ at the point $x=(r,0,...,0)$ with multiplicities 1 and $(n-1)$, respectively. Now, as we saw in the Introduction, $S_{1}^{ij}(D^2 u)=\delta_{ij}$. Note that in the case $k=2$ we can use 
the formula $S_2(D^2 u)=\frac{1}{2}\left((\Delta u)^2-\abs{D^2 u}^2\right)$ to obtain $S_2^{ij}(D^2 u)$. Indeed, we have
\begin{equation*}
S_2(D^2 u)=\frac{1}{2}\left((\Delta u)^2-\abs{D^2 u}^2\right)=\frac{1}{2}\left(\left(\sum u_{jj}\right)^2-\sum u_{ij}^2\right),
\end{equation*}
\begin{equation*}
\begin{split}
S_2^{ij}(D^2 u)&=S_1(D^2 u)\delta_{ij}-u_{ij}=\left(n\lambda_2+(\lambda_1-\lambda_2)\right)\delta_{ij}-\delta_{ij}\lambda_2-\frac{x_{i}x_{j}}{\abs{x}^2}\left(\lambda_1-\lambda_2\right)\\
&=(n-1)\lambda_2\delta_{ij}+\left(\lambda_1-\lambda_2\right)\left(\delta_{ij}-\frac{x_{i}x_{j}}{\abs{x}^2}\right).
\end{split}
\end{equation*}

To unify all the cases for $k\in\{1,...,n\}$, we use the radial form of the $k$-Hessian operator \eqref{Radial:Hess}. For simplicity we denote $S_k=S_k(D^2 u)$ and $S_{k}^{ij}=S_{k}^{ij}(D^2 u)$. Thus, we rewrite $S_{k}=c_{n,k}\lambda_2^{k-1}\left(n\lambda_2+k\left(\lambda_1-\lambda_2\right)\right)$ depending on the eigenvalues of $D^2 u$.

Then 
\begin{equation}\label{eq:Sij}
S_{k}^{ij}=\frac{\partial S_{k}}{\partial u_{ij}}=\frac{\partial S_{k}}{\partial\lambda_1}\frac{\partial\lambda_1}{\partial u_{ij}}+\frac{\partial S_{k}}{\partial\lambda_2}\frac{\partial\lambda_2}{\partial u_{ij}},
\end{equation}	
where
\begin{equation}\label{eq:Skl1l2}
\frac{\partial S_{k}}{\partial\lambda_1}=kc_{n,k}\lambda_2^{k-1}\mbox{ and }
\frac{\partial S_{k}}{\partial\lambda_2}=kc_{n,k}\lambda_2^{k-2}\left((n-1)\lambda_2+(k-1)(\lambda_1-\lambda_2)\right).
\end{equation}

From the equality $\Delta u=n\lambda_2+(\lambda_1-\lambda_2)$ we deduce that
\begin{equation}\label{eq:l1l2uij}
(n-1)\frac{\partial\lambda_2}{\partial u_{ij}}+\frac{\partial\lambda_1}{\partial u_{ij}}=\delta_{ij}.
\end{equation}

Thus, from \eqref{eq:Sij}-\eqref{eq:l1l2uij}, we have
\begin{equation}\label{eq:Sijgen}
\begin{split}
S_k^{ij}&=kc_{n,k}\lambda_2^{k-1}\frac{\partial\lambda_1}{\partial u_{ij}}+kc_{n,k}\lambda_2^{k-2}\left((n-1)\lambda_2+(k-1)(\lambda_1-\lambda_2)\right)\frac{\partial\lambda_2}{\partial u_{ij}}\\
&=kc_{n,k}\lambda_2^{k-2}\left(\left(\frac{\partial\lambda_1}{\partial u_{ij}}+(n-1)\frac{\partial\lambda_2}{\partial u_{ij}}\right)\lambda_2+(k-1)(\lambda_1-\lambda_2)\frac{\partial\lambda_2}{\partial u_{ij}}\right)\\
&=kc_{n,k}\lambda_2^{k-2}\left(\lambda_2\delta_{ij}+(k-1)(\lambda_1-\lambda_2)\frac{\partial\lambda_2}{\partial u_{ij}}\right).  
\end{split}
\end{equation}

In particular
\begin{equation}\label{eq:Sijigual}
S_k^{ij}=kc_{n,k}\lambda_2^{k-1}\delta_{ij},
\end{equation}
when $\lambda_1=\lambda_2$.

Now, from \eqref{eq:uij} it follows that
\begin{equation*}
\sum u_{ij}^2=n\lambda_2^2+2\lambda_2(\lambda_1-\lambda_2)+(\lambda_1-\lambda_2)^2.
\end{equation*}

Differentiating the preceding equality with respect to $u_{ij}$ and using \eqref{eq:l1l2uij}, we have
\begin{equation}\label{eq:l1uij}
\frac{\partial\lambda_1}{\partial u_{ij}}=\frac{x_ix_j}{\abs{x}^2},
\end{equation}
provided that $\lambda_1\neq\lambda_2$.

Consequently, from \eqref{eq:l1l2uij}, \eqref{eq:Sijgen} and \eqref{eq:l1uij} we obtain
\begin{equation}\label{eq:Sijdes}
\begin{split}
S_k^{ij}
&=kc_{n,k}\lambda_2^{k-2}\left(\lambda_2\delta_{ij}+\left(\frac{k-1}{n-1}\right)\left(\lambda_1-\lambda_2\right) \left(\delta_{ij}-\frac{x_ix_j}{\abs{x}^2}\right)\right).  
\end{split}
\end{equation}

Let $S$ be the matrix whose entries are $S_k^{ij}$. Then using \eqref{eq:Sijigual} and \eqref{eq:Sijdes} we can write $S$ as
\begin{equation}\label{eq:Sijgenmatrix}
S=kc_{n,k}\lambda_2^{k-2}\left(\lambda_2I_n+\left(\frac{k-1}{n-1}\right)\left(\lambda_1-\lambda_2\right) \left(I_n-\frac{x^Tx}{\abs{x}^2}\right)\right),
\end{equation}
where $I_n$ denotes the identity matrix of order $n$ and $x^Tx$ denotes a square matrix whose entries are given by $\left(x^Tx\right)_{ij}=x_ix_j$. 

Finally, from \eqref{eq:Sijgenmatrix} we have
\begin{equation}\label{eq:wSvT}
\begin{split}
wSv^T&=kc_{n,k}\lambda_2^{k-2}\left\{\lambda_2\left(\frac{x}{\abs{x}},w\right)\left(\frac{x}{\abs{x}},v\right)\right.+\\
&+\left.\left(\lambda_2+\left(\frac{k-1}{n-1}\right)\left(\lambda_1-\lambda_2\right)\right)\left(\left(w,v\right)-\left(\frac{x}{\abs{x}},w\right)\left(\frac{x}{\abs{x}},v\right)\right) \right\},
\end{split}
\end{equation}
for all vectors $w,v\in\RR^n$, where $(\cdot,\cdot)$ is the Euclidean inner product. 

\section{Proof of the main results}
In this section we will prove the Theorems \ref{theo:solweak} and \ref{theo:esential}, Lemma \ref{equiv} and Corollaries \ref{coro:weakC2} and \ref{coro:esential}.

\begin{proof}[Proof of Theorem \ref{theo:solweak}]

Let $\eta\in C_c^1(B_1)$ (not necessarily radial). We consider the spherical averages of $ \eta$, i.e., the radial function
\[
\xi(r):=\frac{1}{\abs{\partial B_1}}\int_{\partial B_1}{\eta(r\theta)\,d\theta}=\fint_{\partial B_1}{\eta(r\theta)\,d\theta},
\]
where, using polar coordinates, for $x\in\RR^n \setminus \{0\}$ we have set $x=r\theta$, with $r=\abs{x} > 0$ and $\theta=x/r\in\partial B_1 =\left\lbrace\theta\in\RR^n:\abs{\theta}=1 \right\rbrace.$

Differentiating with respect to $r$, we get
\[
\xi_r=\fint_{\partial B_1}{\left(\theta,\D\eta\right)\,d\theta}.
\]

Replacing $w$ by $x$ and $v$ by $\D\eta$ in \eqref{eq:wSvT} and using $\D u=\lambda_2 r\theta$ (see \eqref{eq:uij}), we have

\begin{equation*}
\begin{split}
\int_{B_1}{c_{n,k}\abs{x}^{-k}(u')^{k}\left(x,\D \xi\right)}&=c_{n,k}\int_{B_1}{\lambda_2^{k}\left(x,\D \xi\right)}\\
&=\abs{\partial B_1}c_{n,k}\int_0^1{r^n\lambda_2^k\,\xi_r\,dr}\\
&=c_{n,k}\int_0^1{\int_{\partial B_1}{\left\lbrace r^n\lambda_2^k\left(\theta,\D\eta\right)\right\rbrace d\theta}dr}\\
&=\int_{B_1}{\frac{1}{k}\sum{u_i\eta_jS_k^{ij}(D^2u)}}.
\end{split}
\end{equation*}

Thus, as $\int_{B_1}{g(u)\xi}=\int_{B_1}{g(u)\eta}$, we have
\begin{equation*}
\int_{B_1}{\left\lbrace c_{n,k}\abs{x}^{-k}(u')^{k}\left(x,\D \xi\right)+g(u)\xi\right\rbrace }=\int_{B_1}{\left\lbrace \frac{1}{k}\sum{u_i\eta_jS_k^{ij}(D^2u)}+g(u)\eta\right\rbrace }.
\end{equation*}

Finally, replacing $w$ and $v$ by $\nabla u$ in \eqref{eq:wSvT}, we obtain
\begin{equation*}
\begin{split}
\int_{B_1}{\left\lbrace \sum{u_iu_jS_k^{ij}(D^2u)}\right\rbrace}&=\int_{B_1}{\left\lbrace kc_{n,k}\lambda_2^{k-1}\abs{\nabla u}^2\right\rbrace}=kc_{n,k}\int_{B_1}{\abs{x}^{1-k}\abs{\D u}^{k+1}}.\\
\end{split}
\end{equation*}

Therefore
\begin{equation*}
u\in\XX\Leftrightarrow u\in L^{k+1}(B_1)\mbox{ and }\int_{B_1}{\left\lbrace \sum{u_iu_jS_k^{ij}(D^2u)}\right\rbrace}<+\infty,
\end{equation*}
which completes the proof.
\end{proof}

\begin{proof}[Proof of Theorem \ref{theo:esential}] 
Let $u$ be a semistable radial solution of \eqref{Eq:f} according to Definition \ref{def:semistablegen}. Then $u$ satisfies \eqref{ineq:propstable}, since for radial perturbations $\varphi$, \eqref{eq:matchalQgen} reduces to \eqref{ineq:propstable} (see \eqref{eq:DetaSDetaT} below). 

Following \cite[Remark 1.7]{MR2476421} and \cite[Lemma 2.5]{FaNa20}, for any $\varphi\in C_c^1(B_1 \setminus \{0\})$ (not necessarily radial) we consider the spherical averages of $ \varphi^2$, i.e., the radial function
\[
\psi(r):=\frac{1}{\abs{\partial B_1}}\int_{\partial B_1}{\varphi^2(r\theta)\,d\theta}=\fint_{\partial B_1}{\varphi^2(r\theta)\,d\theta},
\]
where, as before, we write $x\in\RR^n \setminus \{0\}$ also as $r\theta$, where $\theta\in\partial B_1$. 

A short computation shows that $ \sqrt{\psi} $ is a Lipschitz continuous function with compact support contained in $(0,+\infty)$, whence $ \xi(x) := \sqrt{\psi}(\vert x \vert)$ can be used as a test function in \eqref{ineq:propstable}. 

Differentiating the last expression with respect to $r$ and using the Cauchy-Schwarz inequality, we obtain
\begin{equation}\label{ineq:xirDeta}
\xi_r^2\leq\fint_{\partial B_1}{\left(\theta,\D\varphi\right)^2\,d\theta}.
\end{equation}

Now, from equality \eqref{eq:wSvT} applied with $w=v=\D\varphi$, we obtain

\begin{equation}\label{eq:DetaSDetaT}
\begin{split}
\sum{\varphi_i\varphi_jS_k^{ij}(D^2u)}&=\D\varphi S\left(\D\varphi\right)^T\\
&=kc_{n,k}\lambda_2^{k-2}\left(\lambda_2\left(\theta,\D\varphi\right)^2\right.+\\
&+\left.\left(\lambda_2+\left(\frac{k-1}{n-1}\right)\left(\lambda_1-\lambda_2\right)\right)\left(\abs{\D\varphi}^2-\left(\theta,\D\varphi\right)^2\right) \right).
\end{split}
\end{equation}

In view of Remark \ref{rema:equiv} and Corollaries \ref{equiv} and \ref{coro:weakC2}, $\lambda_2$ and $n\lambda_2+k(\lambda_1-\lambda_2)$ are positive in the range $0<r<1$. Moreover, $0\leq n(k-1)/(n-1)\leq k$ for all $k\in\{1,...,n\}$. Therefore, regardless of the order of the eigenvalues $\lambda_1$ and $\lambda_2$, 
we conclude that 
\begin{equation*}
\begin{split}
\lambda_2+\left(\frac{k-1}{n-1}\right)\left(\lambda_1-\lambda_2\right)&=\frac{1}{n}\left(n\lambda_2+\frac{n(k-1)}{n-1}\left(\lambda_1-\lambda_2\right)\right)> 0. 
\end{split}
\end{equation*}

Combining the above with the Cauchy-Schwarz inequality, namely $\abs{\D\varphi}^2-\left(\theta,\D\varphi\right)^2\geq 0$, we deduce from \eqref{eq:DetaSDetaT} that
\begin{equation*}
\sum{\varphi_i\varphi_jS_k^{ij}(D^2u)}\geq kc_{n,k}\lambda_2^{k-1}\left(\theta,\D\varphi\right)^2.
\end{equation*}

Therefore, from \eqref{ineq:xirDeta}, we obtain
\begin{equation*}
\begin{split}
\int_{B_1}{\left\lbrace kc_{n,k}\lambda_2^{k-1}\abs{\xi_r}^2+g'(u){\xi}^2\right\rbrace }&\leq\int_{B_1}{\fint_{\partial B_1}{\left\lbrace kc_{n,k}\lambda_2^{k-1}\left(\theta,\D\varphi\right)^2+g'(u){\varphi}^2\right\rbrace }}\\
&\leq\int_{B_1}{\left\lbrace\sum{\varphi_i\varphi_jS_k^{ij}(D^2u)}+g'(u)\varphi^2\right\rbrace }\\
&=\mathcal{Q}_u(\varphi).
\end{split}
\end{equation*}

Finally, the semistability of $u$ follows from a standard density argument. The proof is now complete.
\end{proof}

\begin{proof}[Proof of Lemma \ref{equiv}]
Let $u$ be a weak radial solution of \eqref{Eq:f}. By Theorem \ref{theo:solweak}, $u$ satisfies
\begin{equation}\label{weakradial}
c_{n,k}\int_{0}^{1}r^{n-k}(u')^k\xi'\,dr=-\int_{0}^{1}r^{n-1}g(u)\xi\,dr,
\end{equation}
for all $\xi\in C_c^1(0,1)$.  
Take $\xi(r)=\int_{r}^{1}\eta (s)ds$, where $\eta$ is smooth and has compact support in $(0,1)$ as a test function in \eqref{weakradial}. After an integration by parts, we have
\begin{equation*}
\begin{split}
c_{n,k}\int_0^1{r^{n-k}(u')^k(-\eta)\,dr}&=-\int_0^1\left(\int_r^1\eta(s)ds\right)r^{n-1}g(u)\,dr\\
&=\int_0^1{\left(\int_0^rs^{n-1}g(u)ds\right)(-\eta)\,d}r,
\end{split}
\end{equation*}
which leads to
\[
c_{n,k}r^{n-k}(u')^k=\int_0^rs^{n-1}g(u)\,ds\mbox{ a.e.}
\]

Conversely, applying H\"older's inequality to $r^{n-k}(u')^{k}\abs{\xi'}$ with exponents $p=(k+1)/k$ and $q=k+1$, we obtain 
\begin{equation*}
\begin{split}
\abs{\int_0^1r^{n-k}(u')^{k}\xi'dr}&\leq\int_0^1{r^{n-k}(u')^{k}\abs{\xi'}\,dr}\\ &\leq\left(\int_0^1r^{n-k}(u')^{k+1}dr\right)^{\frac{k}{k+1}}\left(\int_0^1r^{n-k}\abs{\xi'}^{k+1}dr\right)^\frac{1}{k+1}<\infty.
\end{split}
\end{equation*}

Thus, if $u$ is an integral radial solution of \eqref{Eq:f}, we can integrate from 0 to 1 the integral solution condition multiplied by $\xi'(r)$, where $\xi\in C_c^1(0,1)$, thus obtaining
\begin{equation}\label{onweak}
\int_0^1{\left(\int_0^rs^{n-1}g(u)ds\right)\xi'\,dr}=c_{n,k}\int_0^1{r^{n-k}(u')^k\xi'\,dr}<\infty.
\end{equation}

Finally, integrating by parts the left-hand side of \eqref{onweak}, we conclude that $u$ is a weak radial solution of \eqref{Eq:f}.
\end{proof}

\begin{proof}[Proof of Corollary  \ref{coro:weakC2}]
Let $u$ be an integral radial solution of \eqref{Eq:f}. By Lemma \ref{equiv}, $u$ is a weak radial solution of \eqref{Eq:f}, and by the Sobolev embedding in one dimension, $u$ is a continuous function of $r=\abs{x}\in [\delta,1]$ for every $\delta\in (0,1)$. Thus $u$ is continuous on $(0,1)$ and therefore so also is $s^{n-1}g(u)$. Now from \eqref{integraldef} we have 
\begin{equation*}\label{derivative}
u'=c_{n,k}^{-\frac{1}{k}}\left(\frac{\int_0^rs^{n-1}g(u)ds}{r^{n-k}}\right)^\frac{1}{k},\,r\in (0,1].
\end{equation*}

It follows directly that
\begin{equation*}\label{eq:derivative}
	u''=\frac{c_{n,k}^{-\frac{1}{k}}}{k}\left(\frac{\int_0^r{s^{n-1}g(u)\,ds}}{r^{n-k}}\right)^{\frac{1-k}{k}}\left(r^kg(u)+\frac{(k-n)\int_0^r{s^{n-1}g(u)\,ds}}{r^{n-k+1}}\right) ,\,r\in (0,1],
	\end{equation*}	
which concludes the proof.
\end{proof}

We claim that if $u$ is a $C^2$ radial function on $B_1$, then $u'\in C_{loc}^{0,1}(B_1)$. To prove our claim, we first observe that $u'\in C^0(B_1)\cap C^1(B_1\setminus\{0\})$ with $u'(0)=0$ and $\abs{u'(x)}\leq C_R\abs{x}$ in any open ball $B_R$ of radius $R>0$ contained in $B_1$. Now
from \eqref{Eq:f} and \eqref{Radial:Hess} we obtain, for $0<r<1$,
\begin{equation*}
kc_{n,k}\left(\frac{u'}{r}\right)^{k-1}\left(\left(\frac{n-k}{k}\right)\left(\frac{u'}{r}\right)+u''\right)=g(u),
\end{equation*}
from which we deduce that
\begin{equation*}\label{eq:d2u}
u''=\left(\frac{g(u)}{kc_{n,k}}\right) \left(\frac{u'}{r}\right)^{1-k}-\left(\frac{n-k}{k}\right)\left(\frac{u'}{r}\right),\,\forall r\in(0,1).
\end{equation*}

Thus, for any $0<R<1$, the function $u''$ is bounded on $B_R\setminus\{0\}$. Also for any $x,y\in B_1$ such that $R>\abs{y}>\abs{x}>0$, we have
\begin{equation*}
\begin{split}
\abs{u'(\abs{y})-u'(\abs{x})}&\leq\int_{\abs{x}}^{\abs{y}}\abs{u''(t)}dt\\
&\leq\sup_{\xi\in [\abs{x},\abs{y}]}\abs{u''(\xi)}(\abs{y}-\abs{x})\\
&\leq\sup_{z\in B_{R}\setminus\{0\}}\abs{u''(z)}\abs{y-x}.
\end{split}
\end{equation*}
We conclude that $u'\in C_{loc}^{0,1}(B_1)$. This will be used in the proof of Lemma \ref{lemm:esential} below.	
\medskip

Note that when $u$ and $\varphi$ are radial functions, the left-hand side of \eqref{eq:matchalQgen} can be written as 
	\begin{equation}\label{eq:radialQ}
	\Q_u(\varphi)=\int_{B_1}{kc_{n,k}\abs{x}^{1-k}\abs{\nabla u}^{k-1}\abs{\nabla\varphi}^2+g'(u)\varphi^2},
	\end{equation}
	which follows directly from \eqref{eq:DetaSDetaT}. Further, we may choose the test functions $\varphi$ to lie in $C^1_c(B_1)$. 
\begin{lemma}\label{lemm:esential}
	Let $u$ be any radial solution of \eqref{Eq:f}. Then
	\begin{equation}\label{eq:Qceta}
	\Q_u(u'\eta)=kc_{n,k}\int_{B_1}{\left(\frac{u'}{\abs{x}}\right)^{k+1}\left\lbrace\abs{\abs{x}\D\eta}^2+\frac{(k-1)\left(x,\D\eta^2\right)}{k+1}-\frac{(2n-k-1)\eta^2}{k+1}\right\rbrace},
	\end{equation}
	for every radially symmetric function $\eta\in (H_c^1\cap L_{loc}^{\infty})\left(B_1\right)$. 
\end{lemma}

\begin{proof}

	Let $\eta\in H_c^1(B_1)\cap L_{loc}^\infty(B_1)$ and $\zeta\in C_{loc}^{0,1}(B_1)$ be radial functions. Then, by a standard density argument, we can take $\varphi=\zeta\eta\in H_c^1(B_1)\cap L_{loc}^\infty(B_1)$ in \eqref{eq:radialQ} to obtain
	\begin{equation}\label{eq:Qzeta}
	\Q_u(\zeta\eta)=\int_{B_1}{kc_{n,k}\abs{\frac{u'}{\abs{x}}}^{k-1}\abs{\D(\zeta\eta)}^2+g'(u)(\zeta\eta)^2}.
	\end{equation}
	
	Thus, as $u$ and $\zeta$ are radial functions, we get 
	\begin{equation}\label{eq:radialzetaeta}
	\abs{\frac{u'}{\abs{x}}}^{k-1}\abs{\D(\zeta\eta)}^2=\lambda_2^{k-1}\left(\zeta^2\abs{\nabla\eta}^2+\left(2\zeta\eta(\nabla\zeta,\nabla\eta)+\eta^2\abs{\nabla\zeta}^2\right)\right),
	\end{equation}
	where $\lambda_2=u'/{\abs{x}}.$
	
	From \eqref{Radial:Hess} and differentiating \eqref{Eq:f} with respect to $r$, we obtain
	\begin{equation*}
	\begin{split}
	c_{n,k}^{-1}g'(u)u'&=-\frac{n-1}{r^2}\left(\frac{u'^k}{r^{k-1}}\right)+\Delta\left(\frac{u'^k}{r^{k-1}}\right)\\
	&=-\frac{(n-1)\lambda_2^k}{r}+\Delta\left(r\lambda_2^k\right),\; r>0.
	\end{split}
	\end{equation*}
	
	Then, multiplying the latter equation by $\eta^2 u'$, integrating by parts and taking into account that $\lambda_2\in L_{loc}^\infty
(B_1)$, we have
\begin{equation*}
\begin{split}
c_{n,k}^{-1}\int_{B_1}g'(u)(u'\eta)^2&=-(n-1)\int_{B_1}\lambda_2^{k+1}\eta^2-\int_{B_1}\left(\nabla\left(r\lambda_2^{k}\right),\nabla(\eta^2 u')\right)\\
&=-\frac{k(2n-k-1)}{k+1}\int_{B_1}\lambda_2^{k+1}\eta^2+\frac{k(k-1)}{k+1}\int_{B_1}\lambda_2^{k+1}\left(x,\nabla\eta^2\right)\\
&-k\int_{B_1}\lambda_2^{k-1}\left(\abs{\nabla u'}^2\eta^2+2u'\eta(\nabla u',\nabla\eta)\right).
\end{split}
\end{equation*}

	Finally, using $\zeta=u'$ in \eqref{eq:Qzeta}, \eqref{eq:radialzetaeta} and the previous equation, we obtain \eqref{eq:Qceta}, which completes the proof. 
\end{proof}

\begin{proof}[Proof of Corollary  \ref{coro:esential}]
If $u$ is a semistable radial solution of \eqref{Eq:f}, then from \eqref{eq:Qceta} it follows that $\Q_u(u'\eta)\geq 0$ for every radially symmetric function $\eta\in (H_c^1\cap L_{loc}^{\infty})\left(B_1\right)$, that is, \eqref{goeto0} holds.

For the converse, let $0<\epsilon<1$, let $\varphi\in C_c^1\left(B_1\right)$ be a radially symmetric function and set $\varphi_{\epsilon,\sigma}=\varphi\varsigma_{\epsilon,\sigma}/u'$, where
\begin{equation}\label{var_eps_sig}
\varsigma_{\epsilon,\sigma}(r):=\begin{cases}
\begin{cases}
0&\mbox{If }0\leq r<\epsilon^2,\\
1-\xi\left(\frac{\log r}{\log\epsilon}\right)&\mbox{If }\epsilon^2\leq r<\epsilon,\\
1&\mbox{If }r\geq \epsilon,
\end{cases}&\mbox{If }n=2\sigma,\\
\begin{cases}
0&\mbox{If }0\leq r<\epsilon,\\
\xi\left(\frac{r}{\epsilon}\right) &\mbox{If }\epsilon\leq r<2\epsilon,\\
1&\mbox{If }r\geq 2\epsilon.
\end{cases}&\mbox{If }n\geq 2\sigma+1,
\end{cases}
\end{equation}
where $\sigma\in \NN$ and $\xi(t)=2\left(1-t\right)^2\left(\frac52-t\right)$.

By applying \eqref{eq:Qzeta} with $\eta = \varphi_{\epsilon,\sigma}$, we conclude that
\begin{equation}\label{ineq:usemvarep}
\begin{split}
0 \le \Q_u(u'\varphi_{\epsilon,\sigma})&=\int_{B_1}{\left\lbrace kc_{n,k}\lambda_2^{k-1}\abs{\D (u'\varphi_{\epsilon,\sigma})}+g'(u)(u'\varphi_{\epsilon,\sigma})^2\right\rbrace }\\
&=\int_{B_1}{\left\lbrace kc_{n,k}\lambda_2^{k-1}\abs{\D (\varphi\varsigma_{\epsilon,\sigma})}+g'(u)(\varphi\varsigma_{\epsilon,\sigma})^2\right\rbrace }\\
&=\int_{B_1}{\left\lbrace\left(kc_{n,k}\lambda_2^{k-1}\abs{\D\varphi}+g'(u)\varphi^2\right)\varsigma_{\epsilon,\sigma}^2\right\rbrace }+\\
&+kc_{n,k}\int_{B_1}{\left\lbrace \lambda_2^{k-1}\left(2\varphi\varphi'\varsigma_{\epsilon,\sigma}\varsigma_{\epsilon,\sigma}'+\varphi^2\varsigma_{\epsilon,\sigma}'^2\right) \right\rbrace }.
\end{split}
\end{equation}	

Now let
\begin{equation}\label{eq:Iepsig}
\begin{split}
I_{\epsilon,\sigma}&:=\fint_{B_1}{\left\lbrace \lambda_2^{k-1}\left(2\varphi\varphi'\varsigma_{\epsilon,\sigma}\varsigma_{\epsilon,\sigma}'+\varphi^2\varsigma_{\epsilon,\sigma}'^2\right) \right\rbrace }\\
&=\int_0^1{\left\lbrace{(u')}^{k-1}\left(2\varphi\varphi'\varsigma_{\epsilon,\sigma}\varsigma_{\epsilon,\sigma}'+\varphi^2\varsigma_{\epsilon,\sigma}'^2\right) \right\rbrace r^{n-k}dr}.
\end{split}
\end{equation}

\begin{itemize}
	\item If $u\in C^2(B_1)$ or $u\in\XX$ and $n<2k$, we take $\sigma=1$ in \eqref{var_eps_sig} (Note that, if $n<2k$ and $u\in\XX$, then $u'(0)=0$ by \cite[Lemma 3.3]{NaSa19}). Then 
	\begin{equation}\label{ineq:Iepsigur0}
	\begin{split}
	\abs{I_{\epsilon,1}}&\leq
	\begin{cases}
	\frac{\epsilon^{k-1}}{-\log\epsilon}\int_{\epsilon^2}^{\epsilon}{\left(16\abs{\varphi\varphi'}r+\left(\frac{64\varphi^2}{-\log\epsilon}\right)\right)r^{n-k-2}dr}&\mbox{If }n=2,\\
	\epsilon^{k-3}\int_{\epsilon}^{2\epsilon}{\left(16\abs{\varphi\varphi'}\epsilon+64\varphi^2\right)r^{n-k}dr}&\mbox{If }n\geq 3,
	\end{cases}\\
	&\leq \epsilon^{n-2}o(1)=o(1),  
	\end{split}
	\end{equation}
	where $I_{\epsilon,1}$ is given by \eqref{eq:Iepsig}.
	
	\item If $u\in\XX$ and $n\geq 2k$, we take $\sigma=k$ in \eqref{var_eps_sig}. Then 
	\begin{equation}\label{ineq:IepsiguW1k}
	\begin{split}
	\abs{I_{\epsilon,k}}&\leq
	\begin{cases}
	\frac{1}{-\log\epsilon}\int_{\epsilon^2}^{\epsilon}{r^{n-k}(u')^{k-1}\left(16\abs{\varphi\varphi'}r^{-1}+\left(\frac{64\varphi^2r^{-2}}{-\log\epsilon}\right)\right)dr}&\mbox{If }n=2k,\\
	\epsilon^{-2}\int_{\epsilon}^{2\epsilon}{r^{n-k}(u')^{k-1}\left(16\abs{\varphi\varphi'}\epsilon+64\varphi^2\right)dr}&\mbox{If }n\geq 2k+1,
	\end{cases}\\
	&\leq\left(\begin{cases}
	\frac{1}{-\log\epsilon}\left(\int_{\epsilon^2}^{\epsilon}{r^{n-k}(u')^{k+1}dr}\right)^{\frac{k-1}{k+1}}&\mbox{If }n=2k,\\
	\epsilon^{-2}\left(\int_{\epsilon}^{2\epsilon}{r^{n-k}(u')^{k+1}dr}\right)^{\frac{k-1}{k+1}}&\mbox{If }n\geq 2k+1,
	\end{cases}\right)\times\\
	&\times\left(\begin{cases}
	16\left(\int_{\epsilon^2}^{\epsilon}{r^{n-k-\frac{k+1}{2}}\abs{\varphi\varphi'}^{\frac{k+1}{2}}dr}\right)^{\frac{2}{k+1}}&\mbox{If }n=2k,\\
	16\epsilon\left(\int_{\epsilon}^{2\epsilon}{r^{n-k-\frac{k+1}{2}}\abs{\varphi\varphi'}^{\frac{k+1}{2}}dr}\right)^{\frac{2}{k+1}}&\mbox{If }n\geq 2k+1,
	\end{cases}\right.+\\
		&+\left.\begin{cases}
	\frac{64}{-\log\epsilon}\left(\int_{\epsilon^2}^{\epsilon}{r^{n-2k-1}\abs{\varphi}^{k+1}dr}\right)^{\frac{2}{k+1}}&\mbox{If }n=2k,\\
	64\left(\int_{\epsilon}^{2\epsilon}{r^{n-2k-1}\abs{\varphi}^{k+1}dr}\right)^{\frac{2}{k+1}}&\mbox{If }n\geq 2k+1.
	\end{cases}\right)\\
	&\leq \epsilon^{n-2k}o(1)=o(1).
	\end{split}
	\end{equation}
\end{itemize}

Hence, from \eqref{ineq:usemvarep}-\eqref{ineq:IepsiguW1k}, for $\sigma\in\{1,k\}$, we have 
\[
\int_{B_1}{\left(kc_{n,k}\lambda_2^{k-1}\abs{\D\varphi}+g'(u)\varphi^2\right)\varsigma_{\epsilon,\sigma}^2}+ o(1) \geq 0.
\]
Letting $\epsilon\to0$, we obtain, using dominated convergence theorem and Theorem \ref{theo:esential}, that $u$ is semistable, as desired.	
\end{proof}

\begin{proof}[Proof of \cref{th_1,thm:est_g,th_2}]
The proofs of \cref{th_1,thm:est_g,th_2} follow by direct applications of Corollary \ref{coro:esential} (which is fundamental in our charactarization of semistable solutions) and the corresponding results in \cite[Theorems 2.5, 2.6 and 2.8]{NaSa19}.
\end{proof}

\section{A family of semistable radial solutions}
In this section we provide, for $n\geq 2k+8$, a large family of semistable radially increasing unbounded solutions of problems of the type \eqref{Eq:f} on $B_1\setminus\{0\}$. This family is similar to that constructed in \cite{NaSa19} for problem \eqref{Eq:P}. 

We start with the following statement, which was proved in \cite{NaSa19}.
\begin{theorem}\label{thm:familysemistable}
	Let $h\in (C^1\cap L^1)(0,1]$ be a nonnegative function and consider
	\[
	V(r)=r^{(k+1)\left(\delta_{n,k}-2\right)+2}\left(1+\int_0^r{h(s)\,ds}\right),\;\; \forall r\in (0,1],
	\]
	where
%
	\begin{equation}\label{eq:deltank}
	\begin{split}
	\delta_{n,k}&=\frac{-(k+1)n+2\sqrt{2(k+1)n-4k}+2k^2+6k}{(k+1)^2}.\\
	\end{split}
	\end{equation}	

	Now define 
	\[
	u'(r):=r^{\frac{k-1}{k+1}}(V(r))^{\frac{1}{k+1}},\;\; \forall r\in (0,1].
	\]	
	
	Then, for all $n\geq 2k+8$, $u$ is a semistable radially increasing unbounded\, $\XX$-solution of a problem of the type \eqref{Eq:P} on $B_1\setminus\{0\}$ with $g\geq0$, where $u$ is any function with radial derivative $u'$.
\end{theorem}

We prove the above statement for problem \eqref{Eq:f}. To this end, we need the following lemma

\begin{lemma}\label{lemm:newhardy}
	Let $\alpha, \beta\in\RR$, and let $V\in C^1\left(B_1\setminus\{0\}\right)$ be a nonnegative radial function such that
	\begin{equation}\label{ineq:Vcond1}
	\alpha\left(rV'+(n-2\beta-\alpha-2)V\right)\geq 0,\,\forall r\in(0,1].
	\end{equation}
	
	Further, assume that 
	\begin{equation}\label{eq:Vcond2}
	\lim_{r\rightarrow 0}r^{n-2}V=0.
	\end{equation}
	
	Then, for every radially symmetric function $\eta\in C^1_c(B_1)$, we have
	\begin{equation}\label{ineq:hardygen}
	\int_0^1{r^{n-3}V\left(\left(r\eta'+\beta\eta\right)^2-\frac{\alpha^2\eta^2}4\right)\,dr}\geq 0.
	\end{equation}	
	
\end{lemma}
\begin{remark}
	If $\alpha=n-2$, $\beta=0$, $V\equiv 1$ and $n\geq 3$ in Lemma \ref{lemm:newhardy}, then we obtain a Hardy-type inequality on the unit ball.
\end{remark}

\begin{proof}
Let $\eta\in C_c^1(B_1)$ be a radially symmetric function. Then
\begin{equation*}
\begin{gathered}
\int_0^1{r^{n-3}V\left (\alpha\eta-t\left(r\eta'+\beta\eta\right)\right )^2\,dr}\geq 0
\end{gathered}
\end{equation*}
for all $t\in\RR$. Extending the above expression, we get the following quadratic inequality for $t$:

\begin{equation*}
\alpha^2\int_0^1{r^{n-3}V\eta^2\,dr}-2\alpha t\int_0^1{r^{n-3}V\left(r\eta\eta'+\beta\eta^2\right)\,dr}+t^2\int_0^1{r^{n-3}V\left(r\eta'+\beta\eta\right)^2\,dr}\geq 0.
\end{equation*}

Integrating by parts above and using \eqref{eq:Vcond2}, we obtain
\begin{equation*}
\begin{gathered}
\alpha^2\int_0^1{r^{n-3}V\eta^2\,dr}+\alpha t\int_0^1{r^{n-3}\left(rV'+\left(n-2\beta-2\right)V\right)\eta^2\,dr}+t^2\int_0^1{r^{n-3}V\left(r\eta'+\beta\eta\right)^2\,dr}\geq 0.
\end{gathered}
\end{equation*}

Thus, the above quadratic inequality is equivalent to
\begin{equation*}
\begin{split}
\left(\alpha\int_0^1{r^{n-3}\left(rV'+\left(n-2\beta-2\right)V\right)\eta^2\,dr}\right)^2 &\leq 4\alpha^2\left(\int_0^1{r^{n-3}V\eta^2\,dr}\right)\times\\
&\times\left(\int_0^1{r^{n-3}V\left(r\eta'+\beta\eta\right)^2\,dr}\right).
\end{split}
\end{equation*}

From this and \eqref{ineq:Vcond1}, we get the desired inequality \eqref{ineq:hardygen}. The proof is now complete. 
\end{proof}

\begin{proof}[Proof of Theorem \ref{thm:familysemistable}]
The proof presented here follows the same lines of the proof of \cite[Theorem 4.1]{NaSa19}, but except for some minor changes due to the condition \eqref{eq:Vcond2} that gives a control of $V$ near 0.
	
	First, since $n\geq 2k+8$, from \eqref{eq:deltank} we obtain
	
	\begin{equation}\label{ineq:nlamb}
	n+(k+1)\left(\delta_{n,k}-2\right)=\frac{2\sqrt{2(k+1)n-4k}+4k}{k+1}-2\geq 6,
	\end{equation}
	which shows that $V\in L^1(B_1)$ and hence that $u\in\XX$.
%

	Thus, by the definition of $u'$, we have
	\begin{equation}\label{eq:deru}
	u'(r)=r^{\delta_{n,k}-1}\left(1+\int_0^r{h(s)\,ds}\right)^{\frac{1}{k+1}}\geq r^{\delta_{n,k}-1},\;\forall r\in (0,1].
	\end{equation}	
	
	We now rewrite \eqref{eq:deltank} as
	\begin{equation*}
	\begin{split}
	\delta_{n,k}
	&=-\frac{\left (\sqrt{2(k+1)n-4k}+2k\right )\left (n-2k-8\right )}{(k+1)(\sqrt{2(k+1)n-4k}+2k+4)}.\\
	\end{split}
	\end{equation*}
	
	Then
	\[
	u(1)-u(r)\geq\begin{cases}
	-\log r, &\text{if }\delta_{n,k}=0\Longleftrightarrow n=2k+8,\\
	\frac{1-r^{\delta_{n,k}}}{\delta_{n,k}},&\text{if }\delta_{n,k}<0\Longleftrightarrow n>2k+8.	
	\end{cases}
	\]
	
	Therefore 
	\[\lim_{r\to 0}{u(r)}=-\infty.\]
	
	Thus, from \eqref{eq:deru} and \eqref{Eq:f} in their radial form, we obtain
	\begin{equation*}\label{eq:SkFIN}
	\begin{split}
	c_{n,k}^{-1}S_k(D^2u)&=r^{1-n}\left (r^{n-k}(u')^k\right )'=r^{k(\delta_{n,k}-2)}\left(1+\int_0^r{h(s)\,ds}\right)^{\frac{k}{k+1}}\times\\
	&\times\left(n+k(\delta_{n,k}-2)+\frac{krh(r)}{(k+1)\left(1+\int_0^r{h(s)\,ds}\right)}\right).	
	\end{split}
	\end{equation*}
	
	Next, since $h\in C^1(0,1]$, we conclude that $u'\in C^2(0,1]$. Thus $0<S_k(D^2u)\in C^1(\overline{B_1}\setminus\{0\})$ (Note that $n+k(\delta_{n,k}-2)=8-\delta_{n,k}\geq 8$ by \eqref{ineq:nlamb}). Hence, taking a nonnegative function $g\in C^1$ such that $g(s) = c_{n,k}^{-1}S_k(D^2u(u^{-1}(s)))$, for $s\in(-\infty,u(1)]$, we conclude that $u$ is a solution of a problem of type \eqref{Eq:f} on $B_1\setminus\{0\}$.
	
	It remains to prove the semistability of $u$. First, from the definition of $V$ and \eqref{ineq:nlamb}, it is easily seen that 
	\begin{equation}\label{eq:Vcond2proof}
	\begin{split}
	r^{n-2}V&=r^{n+(k+1)\left(\delta_{n,k}-2\right)}\left(1+\int_0^r{h(s)\,ds}\right)\rightarrow 0\; \mbox{ as}\;\; r\rightarrow 0,
	\end{split}
	\end{equation}
	and 
	\begin{equation}\label{eq:Vcond1proof}
	rV'(r)-\left((k+1)\left(\delta_{n,k}-2\right)+2\right) V(r)=r^{(k+1)\left(\delta_{n,k}-2\right)+3}h(r)\geq 0,\, \forall r\in (0,1].
	\end{equation}
	
	Next define
	\begin{equation}\label{eq:examalpbet}
	\alpha=\frac{2\sqrt{2(k+1)n-4k}}{k+1}\mbox{ and } \beta=\frac{k-1}{k+1}.
	\end{equation}
	
	Then we can write the left-hand side of \eqref{ineq:nlamb} in the form
	
	\begin{equation*}
	n+(k+1)\left(\delta_{n,k}-2\right)=\alpha+2\beta.
	\end{equation*}
		
	Since $n\geq2k+8$, we have $\alpha\geq4(k+2)/(k+1)>0$. Then from \eqref{eq:Vcond1proof}, \eqref{eq:examalpbet} and the last equality, we obtain
	
	
	\begin{equation}\label{eq:Vcond0proof}
	\begin{split}
	\alpha\left(rV'+(n-2\beta-\alpha-2)V\right)&\geq \alpha\left(n+(k+1)\left(\delta_{n,k}-2\right)-2\beta-\alpha\right)V=0.\\
	\end{split}
	\end{equation}
	
	Now, by a simple calculation, we have
	\begin{equation*}
	\frac{\alpha^2}{4}-\beta^2=\frac{2n-k-1}{k+1}.
	\end{equation*}
	
	Finally, from the above, \eqref{ineq:Vcond1}, \eqref{eq:Vcond2proof}, \eqref{eq:Vcond0proof}, Lemma \ref{lemm:newhardy} and Corollary \ref{coro:esential}, we conclude that $u$ is semistable. The proof of the theorem is now complete.
\end{proof}

\begin{proposition}\label{prop:inequrr}
Let $\{r_m\}\subset(0, 1]$, $\{M_m\}\subset\RR^{+}$ be two sequences with $r_m\downarrow0$. Then, for $n\geq 2k+8$, there exists a $u\in\XX$, which is a semistable radially increasing unbounded solution of a problem of type \eqref{Eq:f} in $B_1\setminus\{0\}$, satisfying

\begin{equation*}\label{ineq:famsemurr}
\abs{u''(r_m)}\geq M_m,\;\;\forall m\in\NN.
\end{equation*}
\end{proposition}

\begin{corollary}\label{cor:inequrr}
Let $n\geq 2k+8$. Then there is no function $\varphi:(0, 1]\to\RR^{+}$ having the following property: for every $u\in\XX$ semistable radially
increasing solution of a problem of type \eqref{Eq:f} on $B_1\setminus\{0\}$ with $g\geq 0$, there exist $C > 0$ and $\epsilon\in(0, 1]$ such that $\abs{u''(r)}\leq C\varphi(r)$ for every $r\in(0,\epsilon]$.
\end{corollary}

\begin{proposition}\label{prop:inequrrr}
Let $\{r_m\}\subset(0, 1]$, $\{M_m\}\subset\RR^{+}$  be two sequences with $r_m\downarrow0$. Then, for $n\geq 2k+8$, there exists a $u\in\XX$, which is a semistable radially increasing unbounded solution of a problem of type \eqref{Eq:f} on $B_1\setminus\{0\}$ with $g'\leq 0$, satisfying

\begin{equation*}\label{ineq:famsemurrr}
\abs{u'''(r_m)}\geq M_m,\; \forall m\in\NN.
\end{equation*}
\end{proposition}
\begin{corollary}\label{cor:inequrrr}
Let $n\geq 2k+8$. Then there is no function $\varphi:(0, 1]\to\RR^{+}$ having the following property: for every $u\in\XX$ semistable radially
increasing solution of a problem of type \eqref{Eq:f} on $B_1\setminus\{0\}$ with $g'\leq 0$, there exist $C > 0$ and $\epsilon\in(0, 1]$ such that $\abs{u'''(r)}\leq C\varphi(r)$ for every $r\in(0,\epsilon]$.
\end{corollary}
\begin{proof}
The proofs of propositions \ref{prop:inequrr} and \ref{prop:inequrrr} and of corollaries \ref{cor:inequrr} and \ref{cor:inequrrr} follow from Theorem \ref{thm:familysemistable}, Corollary \ref{coro:esential} and \cite[propositions 4.3 and 4.4, corollaries 4.4 and 4.5]{NaSa19}.
\end{proof}

\section{Additional comments}
For the convenience of the reader, and in order to make explicit the constant $c_{n,k}$ which intervenes in the radial form of the $k$-Hessians, we give here a detailed proof of \eqref{Radial:Hess}. Consider a radial function $u$ on the unit ball $B_1$ of $\RR^n$, that is, $u(x)=v(r)$ where $r=\abs{x}=\sqrt{x_1^2+\ldots+x_n^2}$.
We compute the Hessian of $v$, rotating $B_1$ so that $x_1 =r$:
\begin{equation}\label{eq:hessianv}
D^2v=
\begin{pmatrix}
v''&0&\cdots&0\\
0&\frac{v'}{r}&\ddots&\vdots\\
\vdots&\ddots&\ddots&0\\
0&\cdots&0&\frac{v'}{r}
\end{pmatrix}.
\end{equation}

Since the Hessian operators are invariant under rotations of coordinates, we have $S_k\left(D^2u\right)=S_k\left(D^2v\right)$. Then we compute the minors of order $k$ of $D^2v$ using an ingenious argument from \cite{Warren16}. Let 
\[
I=\left\lbrace \alpha\subset\{1,...,n\}:\; \abs{\alpha}=k\right\rbrace ,
\]
and let $A=\{\alpha\in I:1\in\alpha\}$ and $B=\{\alpha\in I:n\in\alpha\}$.

We express $I$ as a disjoint union

\begin{equation}\label{eq:Iradform}
I=(A\cap B)\cup(B\setminus A)\cup(I\setminus B).
\end{equation}

Now define
\[
S_{k}^{(\alpha)}=\mbox{det}
\begin{pmatrix}
k\times k\mbox{ matrix of \eqref{eq:hessianv}}
\\
\mbox{with row and columns}
\\
\mbox{chosen from $\alpha$}
\end{pmatrix}.
\]

For $\alpha\in I$, we need to consider three cases:
\begin{enumerate}[$\mbox{Case }1:$]
	\item If $\alpha\in (A\cap B)$,
	\[
	S_{k}^{(\alpha)}=v''\left(\frac{v'}{r}\right)^{k-1} \mbox{ and }\abs{A\cap B}=\binom{n-2}{k-2}.
	\]
	\item If $\alpha\in (B\setminus A)$,
	\[
	S_{k}^{(\alpha)}=\left(\frac{v'}{r}\right)^{k}\mbox{ and }\abs{B\setminus A}=\binom{n-1}{k}.
	\]
	\item If $\alpha\in (I\setminus B)$,
	\[
	S_{k}^{(\alpha)}=v''\left(\frac{v'}{r}\right)^{k-1}\mbox{ and }\abs{I\setminus B}=\binom{n-2}{k-1}.
	\]
\end{enumerate}

From \eqref{eq:Iradform} and the binomial identities $\binom{n-2}{k-2}+\binom{n-2}{k-1}=\binom{n-1}{k-1}=\frac{k}{n}\binom{n}{k}$ and $\binom{n-1}{k}=\frac{n-k}{n}\binom{n}{k}$, we have
\[
\begin{split}
S_k(D^2v)&=\sum_{\alpha\in I }{S_{k}^{(\alpha)}}=\binom{n-2}{k-2}v''\left(\frac{v'}{r}\right)^{k-1}+\binom{n-1}{k}\left(\frac{v'}{r}\right)^{k}+\binom{n-2}{k-1}v''\left(\frac{v'}{r}\right)^{k-1}\\
&=c_{n,k}\left(\frac{v'}{r}\right)^{k-1}\left((n-k)\left(\frac{v'}{r}\right)+kv''\right)=c_{n,k}r^{1-n}\left(r^{n-k}(v')^k\right)',
\end{split}
\]
where the last equality shows that we must have $c_{n,k}=\binom{n}{k}/n$. This completes the proof of the formula in \eqref{Radial:Hess}.

\section*{Acknowledgements}
M. Navarro was supported by XUNTA de Galicia under Grant Axudas \'{a} etapa de formaci\'{o}n posdoutoral 2017 and partially supported by AEI of Spain under Grant MTM2016-75140-P and co-financed by European Community fund FEDER and XUNTA de Galicia under grants GRC2015-004 and R2016/022. 

\bibliographystyle{plain}
\bibliography{k-Hessianbib}
\end{document}